\documentclass{amsart}
\usepackage{amsthm}
\usepackage{amssymb}
\usepackage[nobysame]{amsrefs}

\usepackage{tikz-cd}
\usepackage{colonequals}
\usepackage{enumerate}
\usepackage{eucal}

\usepackage{hyperref}
\hypersetup{%
  bookmarksnumbered=true,%
  colorlinks=true,%
  linkcolor=blue,%
  citecolor=blue,%
  filecolor=blue,%
  menucolor=blue,%
  urlcolor=blue,%
  bookmarksopen=true,%
  bookmarksdepth=2,%
  pageanchor=true}

\makeatletter
\@namedef{subjclassname@2020}{\textup{2020} Mathematics Subject Classification}
\makeatother

\numberwithin{equation}{section}

\swapnumbers

\theoremstyle{plain}
\newtheorem{theorem}[equation]{Theorem}
\newtheorem*{main}{Theorem}
\newtheorem{proposition}[equation]{Proposition}
\newtheorem{lemma}[equation]{Lemma} 
\newtheorem{corollary}[equation]{Corollary}

\theoremstyle{definition}

\newtheorem{chunk}[equation]{}

\theoremstyle{remark}

\newtheorem{question}[equation]{Question}
\newtheorem*{ack}{Acknowledgements}

\newcommand{\ext}{\operatorname{Ext}}
\newcommand{\domdim}{\operatorname{dom\,dim}}
\newcommand{\codomdim}{\operatorname{codom\,dim}}
\newcommand{\fpdim}{\operatorname{fp\,dim}}
\newcommand{\fidim}{\operatorname{fi\,dim}}
\newcommand{\ginjdim}{\operatorname{Ginj\, dim}}
\newcommand{\gpdim}{\operatorname{Gproj\, dim}}
\newcommand{\gldim}{\operatorname{gl\,dim}}
\newcommand{\injdim}{\operatorname{inj\,dim}}
\newcommand{\pdim}{\operatorname{proj\,dim}}
\newcommand{\Hom}{\operatorname{Hom}}
\newcommand{\RHom}{\operatorname{RHom}}
\newcommand{\rmod}{\operatorname{mod}}
\newcommand{\op}[1]{{#1}^{\mathrm{op}}}

\newcommand{\lra}{\longrightarrow}

\title[Jacobson radical]{Homological dimensions of the\\ Jacobson radical}
\author{Xiao-Wu Chen, Srikanth B. Iyengar, and Ren\'{e} Marczinzik}

\address{Xiao-Wu Chen\\
School of Mathematical Sciences\\
University of Science and Technology of China\\
Hefei 230026\\
Anhui, PR China}

\address{Srikanth B. Iyengar\\ 
Department of Mathematics\\
University of Utah\\ 
Salt Lake City, UT 84112\\ 
U.S.A.}

\address{Ren\'{e} Marczinzik\\ 
Mathematical Institute of the University of Bonn\\ 
University of Bonn\\ 
Endenicher Allee 60, 53115 Bonn \\
Germany.}

\begin{document}

\begin{abstract} 
This work presents results on the finiteness, and on the symmetry properties, of various homological dimensions associated to the Jacobson radical and its higher syzygies, of a semiperfect ring. 
\end{abstract}



\date{\today}

\maketitle

\section{Introduction}
The focus of this work is on the homological properties of the Jacobson radical, and its higher syzygies, over a semiperfect noetherian ring. Fix such a ring $A$, with Jacobson radical $J$; to avoid trivial considerations, we assume  $A$ is not semisimple, equivalently, $J\ne 0$.  This family contains local rings, and also the class of  finite algebras over commutative noetherian semilocal rings that are complete with respect to the $J$-adic topology, and, in particular, Artin algebras.

It is well-understood that the homological invariants of the $A$-module $A_0\colonequals A/J$ capture properties of the ring $A$ itself. As far as invariants derived from projective resolutions are concerned,  the same is true also of $J$, for it is the first syzygy module of $A_0$.  For instance since $A$ is not semisimple one has
\[
\pdim_AJ = \pdim_AA_0-1 = \gldim A - 1\,.
\]
Therefore the projective dimension of $J$ is finite if and only if  the projective dimension of $A_0$ is finite, and this holds precisely when $A$ has finite global dimension.   The same holds for the higher syzygies of $A_0$. The situation is different for invariants derived from injective resolutions for one expects the properties of $A$ to intervene. For instance, the finiteness of $\injdim_AJ$ does not, a priori, imply that of $\injdim_AA_0$, unless $A$ itself has finite injective dimension. Nevertheless we prove:

\begin{main}[see~\ref{thm:inj-dim}]
\label{main:injdim}
For any semiperfect noetherian ring $A$ one has 
\[
\injdim_AJ =\gldim A\,.
\]
\end{main}

The equality above holds even when $A$ is semisimple, for then both invariants involved are zero. When $A$ is also commutative, the theorem above is contained in the work of Ghosh, Gupta, and Puthenpurakal~\cite{Ghosh/Gupta/Puthenpurakal:2018}.
In \cite{Marczinzik:2020} the third author established  this equality  for various classes of Artin algebras, and conjectured that it  holds for  Artin algebras. The theorem above confirms this hunch.  We deduce it from the more general statement that for any finitely generated $A$-module $M$, one has  $\ext_A^i(M,J)=0$ for $i\gg 0$ if and only if $\pdim_AM$ is finite; see Proposition~\ref{prp:pd}. 

We explore two possible generalizations of the theorem above. One concerns the injective dimension of  $\Omega^n(A_0)$ for $n\ge 2$, the higher syzygies of $A_0$.  What little we could prove is recorded in Proposition~\ref{prp:syz-question} and Theorem~\ref{thm:AG-algebras}; see also the discussion around Question~\ref{question:syz}. The work of Gelinas in \cite{Gelinas:2022} is one motivation for pursuing this line of enquiry.  

The other direction we pursue stems from a symmetry property that is a direct corollary of the theorem above: Since the global dimension of  $A$ equals that of  its opposite ring $\op A$,  the injective dimension of $J$ as a left $A$-module equals its injective dimension as a right $A$-module. It is natural to ask whether this is true for other homological invariants of $J$. The most decisive result we offer in this direction is that when $A$ is a semilocal Noether algebra the Gorenstein projective dimension of $J$  on the left and on the right coincide, and that this number is finite precisely when $A$ is Iwanaga-Gorenstein; see Theorem~\ref{thm:gpd} and Corollary~\ref{cor:gpd}.  The corresponding statement for Gorenstein injective dimensions remains open, except when $A$ is a commutative semilocal noetherian ring; in this case the symmetry property is clear, and the key conclusion is that $\ginjdim_AJ=\injdim A$; see Proposition~\ref{prp:commutative}. When $A$ is an Artin algebra this question is closely connected to the Gorenstein symmetry conjecture; see Proposition~\ref{prp:gsc}.

\begin{ack}
This work is partly supported by  the National Natural Science Foundation of China grants 12131015 and 12161141001 (XWC), the United States
National Science Foundation grant DMS-2001368 (SBI), and a grant from the Deutsche Forschungsgemeinschaft, project number 428999796 (RM).
The authors thank the Mathematisches Forschungs Institut, Oberwolfach, for hosting us during a workshop during which this project was born.
\end{ack}

\section{Semiperfect noetherian rings}
Throughout $A$ is a noetherian ring with a unit and $J$  its Jacobson radical. The standing hypothesis is that $A$ is 
\emph{semiperfect}, that is to say, each finitely generated $A$-module (either left or right) has a projective cover. This class of rings includes Artin algebras, local rings, and algebras finite over commutative noetherian complete semilocal rings; see \cite[\S23 and \S24]{Lam:1991}.

Unless stated otherwise, we consider only left modules. The \emph{top} of an $A$-module $M$ is the quotient module $M_0\colonequals M/JM$. 
In what follows the following exact sequence of $A$-modules is required often:
\begin{align}
\label{equ:can}
0\lra J \lra A \xrightarrow{\ \pi\ } A_0\lra 0\,.
\end{align}
Note that $A_0$ is a semisimple ring.

\begin{lemma} 
\label{lem:proj}
Let $A$ be a semiperfect ring and $M$ a finitely generated $A$-module. If the surjection $M\rightarrow M_0$ factors through a projective $A$-module, then $M$ is projective.
\end{lemma}

\begin{proof}
Let $t\colon M\to M_0$ denote the natural surjection and let $p\colon P\to M_0$ be a projective cover of $M_0$, which exists because $A$ is semiperfect. The hypothesis implies that $t$ factors through $p$ so there is an $A$-module morphism $s\colon M\rightarrow P$ satisfying $t=p\circ s$. Since $p\circ s$ is surjective and $p$ is a projective cover,  $s$ is  surjective, and hence a split epimorphism. Since $s$ induces an isomorphism between the tops $M_0$ and $P_0$,  it is  an isomorphism.
\end{proof}

\begin{proposition}
\label{prp:pd}
Let $A$ be a semiperfect noetherian ring, $M$ a finitely generated $A$-module, and $d$ a nonnegative integer. The following conditions are equivalent:
\begin{enumerate}[\quad\rm(1)]
\item $\pdim_AM\leq d$;
\item $\ext_A^{d+1}(M, J)=0$;
\item $\ext_A^d(M, \pi)\colon \ext_A^d(M, A)\to \ext_A^d(M, A_0)$ is surjective.
\end{enumerate}
Consequently one has equalities
\begin{align*}
\pdim_AM &=\inf\{i\ge 0 \mid \ext_A^{i+1}(M, J)=0\} \\
	&=\sup\{0,i \mid \ext_A^{i}(M, J)\ne 0\}\,.
\end{align*}
\end{proposition}

\begin{proof}
The implication (1)$\Rightarrow$(2) is trivial, whereas (2)$\Rightarrow$(3) is immediate once we apply $\Hom_A(M, -)$ to \eqref{equ:can}.  
The semiperfection of $A$ is not relevant so far but is used in proving (3)$\Rightarrow$(1), for that condition implies that $M$ has a minimal projective resolution. Truncating after the first $d$ steps in such a resolution yields a complex
\[
0\lra K\xrightarrow{\ \iota\ } P_{d-1} \lra \cdots \lra P_1 \lra  P_0\lra  M\lra 0\,,
\]
with each $P_i$ a finitely generated projective, and $K$ the $d$-th syzygy of $M$.  Thus an element in $\ext_A^d(M, A_0)$ is represented by an $A$-module morphism $\xi\colon K\to A_0$.

For any such $\xi$ the surjectivity of $\ext_A^d(M, \pi)$ means that there exists some $A$-module morphism $\nu \colon K\rightarrow A$ such that $\xi$ and $\pi \circ \nu$ represent the same element in $\ext_A^d(M, A_0)$. In other words, $\xi-\pi\circ \nu$ factors through $\iota\colon K\rightarrow P_{d-1}$. It follows that $\xi$ factors through the projective module $A\oplus P_{d-1}$. Then any $A$-module morphism $K\rightarrow S$, with $S$ a semisimple module,  factors through a projective module. In particular, the canonical projection $K\rightarrow K_0$ factors through a projective module, so Lemma~\ref{lem:proj} implies $K$ is projective, as desired.
\end{proof}

We record a couple of remarks concerning the preceding result.

\begin{chunk}
\label{ch:gl=pd}
As $A$ is semiperfect and noetherian $\pdim_A(A_0)=\gldim A $; see \cite[Theorem~12]{Eilenberg:1956}, also \cite[Proposition~2.2]{Ingalls/Paquette:2017}. Assume $\gldim A =\infty$, so $\pdim_A(A_0)=\infty$. Proposition~\ref{prp:pd} implies that for any integer $d\ge 0$, the map
\[
\ext_A^d(A_0, \pi)\colon \ext_A^d(A_0, A) \lra \ext_A^d(A_0, A_0)
\]
is not surjective, which adds to the well-known fact that $\ext_A^d(A_0, A_0)\neq 0$.
\end{chunk}

\begin{chunk}
\label{ch:complexes}
When $M$ is an $A$-complex with $\mathrm{H}_*(M)$ finitely generated, the projective dimension of $M$---see \cite[Definition 2.1.P]{Avramov/Foxby:1991}---can be calculated from the vanishing of $\ext_A^i(M,J)$ as in   Proposition~\ref{prp:pd}, with the proviso that $d 
\ge \sup\{i\mid \mathrm{H}_i(M)\ne 0\}$ holds. In particular, $M$ is quasi-isomorphic to a bounded complex of finitely generated projective modules if and only if $\ext^i_A(M,J)=0$ for $i\gg 0$.
\end{chunk}

\subsection*{Injective dimension}
Let $\op A$ denote the opposite algebra of $A$; thus when $M$ is an $A$-bimodule $\injdim_{\op A}(M)$ is the injective dimension of $M$ as a right $A$-module.

\begin{theorem}
\label{thm:inj-dim}
For any semiperfect noetherian ring $A$ there is an equality
\[
\injdim_AJ=\gldim A \,.
\]
In particular $\injdim_AJ=\injdim_{\op A}J$.
\end{theorem}

\begin{proof}
As noted in ~\ref{ch:gl=pd}, there is an equality $\gldim A =\pdim_A(A_0)$ so Proposition~\ref{prp:pd} yields
\[
\gldim A =\inf\{i\ge 0\mid \ext_A^{i+1}(A_0, J)= 0\}\,.
\]
Thus $\gldim A \leq \injdim_AJ$. The reverse inequality  is clear. As to the last assertion, it remains to recall that $\gldim A =\gldim(\op A)$, and also that the Jacobson radical of $A$ and $\op A$ coincide.
\end{proof}
 
\subsection*{Gorenstein injective dimension}
Let $A$ be a noetherian ring.  The Gorenstein projective dimension and the Gorenstein injective dimension of an $A$-module $M$ are denoted  $\gpdim_AM$ and $\ginjdim_AM$, respectively; see \cite{Holm:2004} for definitions. 

The ring $A$ is said to be \emph{Iwanaga-Gorenstein} if it is noetherian on both sides, and $\injdim_AA$ and $\injdim_{\op A}A$ are both finite; in this case the injective dimensions are equal, by the theorem of Zaks~\cite[Lemma~A]{Zaks:1969}. The ring $A$ is \emph{$d$-Gorenstein} for some integer $d$, if it is Iwanaga-Gorenstein and $\injdim_AA\leq d$.

\begin{lemma}
\label{lem:Gid}
Let $A$ be a semiperfect noetherian ring and $M$ an $A$-module. The following statements hold.
\begin{enumerate}[\quad\rm(1)]
\item If $\injdim_AM<\infty$, then $\ginjdim_AM=\injdim_AM$.
\item If $A$ is $d$-Gorenstein, then $\ginjdim_AM\leq d$.
\end{enumerate}
\end{lemma}

\begin{proof}
Statement (1) can be proved along the lines of \cite[Proposition~2.27]{Holm:2004} while (2) is contained in \cite[Theorem~2.22]{Holm:2004}.
\end{proof}

\begin{proposition}
\label{prop:Gid-lr}
Let $A$ be a semiperfect noetherian ring. When $\gldim A$ is finite there are equalities
\[
\ginjdim_AJ=\injdim_AA=\injdim_{\op A}A=\ginjdim_{\op A}J\,.
\]
\end{proposition}

\begin{proof}
Since $A$ has finite global dimension one has equalities
\[
\injdim_AA=\gldim A=\injdim_{\op A}A \,.
\]
As $\ginjdim_AJ=\injdim_{A}J$, by Lemma~\ref{lem:Gid}(1), applying Theorem~\ref{thm:inj-dim} yields the desired equalities.
\end{proof}

Based on the results above, we raise the following questions: When $A$ is a semiperfect noetherian ring, do the following equalities hold:
\begin{align}
\label{conj:1}
\ginjdim_AJ&=\injdim_AA, \text{ and } \\
\label{conj:2}
\ginjdim_AJ&=\ginjdim_{\op A}J\,?
\end{align}

It follows from Proposition~\ref{prp:gsc} when $A$ is artinian these equalities are consequences of the  Gorenstein symmetry conjecture. In particular, they hold when $A$ is an artinian Gorenstein ring. Equality \eqref{conj:1} holds when $A$ is a semilocal commutative ring; see Proposition~\ref{prp:commutative}.  The analogue of \eqref{conj:2} for Gorenstein projective dimension is also open. In Theorem~\ref{thm:gpd} we prove it for semilocal Noether algebras.

\subsection*{Higher syzygies of $A_0$}
In what follows we write $\Omega^n(M)$ for the $n$-th syzygy of a finitely generated $A$-module $M$; in particular, $\Omega^0(M)=M$. We can speak of ``the" $n$-th syzygy because it is well-defined, since projective covers exist.

\begin{question}
\label{question:syz}
Let $A$ be a semiperfect noetherian ring. If $\Omega^n(A_0) \neq 0$ is then $\injdim_A\Omega^n(A_0)=\gldim A $?
\end{question}

Gelinas~\cite{Gelinas:2022} proves that the big finitistic injective dimension of an Artin algebra $A$ is bounded above by the least integer $n\ge 0$ for which $\Omega^n(A_0)$ is an $(n+1)$-th syzygy. It is an open problem whether this number is finite for  Artin algebras. It is closely connected to properties of the injective resolution of $\Omega^n(A_0)$.

Question~\ref{question:syz} has a positive answer when $n=0$ since $\injdim_A({A_0})=\gldim A $, by \cite[Theorem~12]{Eilenberg:1956}. The case that $n=1$ is precisely Theorem~\ref{thm:inj-dim}.  We also have a positive answer when $A$ is a commutative local ring; this is by a result of Ghosh, Gupta, and Puthenpurakal~\cite[Theorem~3.7]{Ghosh/Gupta/Puthenpurakal:2018}.  Using the computer algebra package QPA~\cite{QPA} we found that  the  equality in question holds for several thousand quiver algebras.

The following observations give further evidence that Question~\ref{question:syz} has a positive answer; they apply, in particular, when $A$ is a $2$-Gorenstein ring.

\begin{proposition} 
\label{prp:syz-question}
Let $A$ be a semiperfect noetherian ring. 
\begin{enumerate}[\quad\rm(1)]
\item If $\gldim A $ is finite, then $\injdim_A\Omega^d(A_0)=\gldim A$ for $d\colonequals \gldim A$.
\item If $A$ is Gorenstein and of infinite global dimension, then  
\[
\injdim_A\Omega^n(A_0)=\infty\quad\text{for  any $n\geq 0$.}
\]
\item If $A$ has finitistic injective dimension at most one, then 
\[
\injdim_A\Omega^n(A_0)=\gldim A \quad\text{whenever $\Omega^n(A_0)\neq 0$.}
\]
\end{enumerate}
\end{proposition}

\begin{proof}
For (1), it suffices to observe  $\pdim_A(A_0)=\gldim A$ and that
\[
\ext_A^d(A_0, \Omega^d(A_0))\neq 0\quad\text{for $d\colonequals\gldim A$.}
\]

(2) Each $\Omega^n(A_0)$ has infinite projective dimension, since $\gldim A$ is infinite. It remains to recall that since $A$ is Gorenstein ring, a finitely generated $A$-module has finite projective dimension if and only if it has finite injective dimension.

(3) We have already observed that the stated equality holds for $n=0, 1$. Assume $n \geq 2$ and that $\Omega^n(A_0)\neq 0$. Since $\ext^i_{A}(A_0, \Omega^i(A_0)) \neq 0$ for each $i$, the injective dimension of $\Omega^n(A_0)$ is at least $n$. Since the finitistic injective dimension is at most one, we infer  $\injdim_A\Omega^n(A_0)=\infty$, which gives the desired equality.
\end{proof}

\section{semilocal noether algebras}
Throughout this section $R$ is a commutative noetherian ring, and $A$ a finite $R$-algebra; in particular, $A$ is noetherian on both sides. 
We call such an $A$ a \emph{Noether} algebra, or a Noether $R$-algebra, if the ring $R$ is to emphasized. The focus is on the case when $R$ is semilocal; then so is $A$; see \cite[Proposition~20.6]{Lam:1991}.

\begin{chunk}
\label{ch:gor-test}
It follows from \cite[Corollary~6.11]{Beligiannis:2000}, see also \cite[Theorem~1.4]{Huang/Huang:2012}, that when $A$ a two-sided noetherian ring and $d\ge 0$ an integer, the following conditions are equivalent:
\begin{enumerate}[\quad\rm(1)]
\item
$A$ is $d$-Gorenstein;
\item
$\gpdim_AM\le d$ for each finitely generated $A$-module $M$;
\item
$\gpdim_AN\le d$ for each finitely generated $\op A$-module $N$.
\end{enumerate}
\end{chunk}

When $A$ is a semilocal Iwanaga-Gorenstein algebra $A$ one has
\[
\gpdim_A(A/J)=\injdim_AA= \gpdim_{\op A}(A/J)\,,
\]
where $J$ is the Jacobson radical of $A$. The result below is a converse. 

\begin{theorem}
\label{thm:gpd}
Let $A$ be a semilocal Noether algebra, with Jacobson radical $J$, and $d\ge 0$ an integer. The following conditions are equivalent
\begin{enumerate}[\quad\rm(1)]
\item The algebra $A$ is $d$-Gorenstein;
\item $\gpdim_A(A/J)\le d$;
\item $\gpdim_{\op A}(A/J)\le d$.
\end{enumerate}
\end{theorem}

\begin{proof}
It suffices to prove that conditions (1) and (2) are equivalent. Since (1)$\Rightarrow$(2) is known,  only the converse is moot.    Given \ref{ch:gor-test}, it suffices to prove 
\[
\gpdim_AM\le d\quad\text{for each $M$ in $\rmod A$.}
\]
Since $A/J$ is a direct sum of simple $A$-modules, and each simple $A$-module occurs in the sum, up to isomorphism, one gets that
\[
\gpdim_Ak\le d \quad\text{for each simple $A$-module $k$.}
\]
In particular, $\ext^i_A(k,A) = 0$ for each such $k$ and $i\ge d+1$ and hence $\injdim_A A\le d$, by \cite[Lemma B.3.1]{Buchweitz:2021}. It thus suffices to prove that $\gpdim_AM$ is finite for each $M\in\rmod A$, for then one has
\[
\gpdim_AM = \max\{i\mid \ext^i_A(M,A)\ne 0\}\,,
\]
and the desired upper bound follows.  

By hypothesis, there exists a  commutative noetherian semilocal ring $R$ such that $A$ is a finite $R$-algebra. We can take $R$ to be the center of $A$, for instance.  We verify the finiteness of $\gpdim_AM$ by an induction on $\dim_RM$, the Krull dimension of $M$ viewed as an $R$-module. The argument is similar to that for \cite[Lemma~B.3.1]{Buchweitz:2021} and goes as follows. Given the upper bound on the G-projective dimension of simple $A$-modules, a standard induction on length yields that $\gpdim_AM\le d$ when the $A$-module $M$ has finite length; equivalently, when $\dim_RM=0$. 

Suppose $\dim_RM\ge 1$. With $\mathfrak{m}$ the Jacobson radical of $R$, consider the  $\mathfrak{m}$-power torsion submodule of $M$, namely, the module
\[
M'\colonequals \{x\in M \mid \mathfrak{m}^n\cdot x=0 \text{ for some $n\ge 0$.}\}
\]
Since $R$ is central in $A$, this is an $A$-submodule of $M$, and of finite length. Thus, given the exact sequence
\[
0\lra M'\lra M\lra \overline{M}\lra 0
\]
it suffices to prove that $\gpdim_A\overline{M}$ is finite. Thus replacing $M$ by $\overline{M}$ one can assume that its $\mathfrak{m}$-power torsion submodule is $0$, equivalently, that there exists an $r\in\mathfrak{m}$ be such that it is not a zero-divisor on $M$; see \cite[Proposition~1.2.1]{Bruns/Herzog:1998}.

We already know $\ext_A^i(M,A)=0$ for $i\ge d+1$, because $\injdim_AA\le d$, so we have only to verify that the natural biduality map is a quasi-isomorphism:
\begin{equation}
\label{eq:ddual}
\theta(M)\colon M \lra \RHom_{\op A}(\RHom_A(M,A),A)\,.
\end{equation}
Equivalently that its mapping cone, $\mathrm{cone}(\theta(M))$ is acyclic; see \cite[(2.3.8)]{Christensen:2000}.

Set $K\colonequals \mathrm{cone}( R\xrightarrow{r}R)$; this is is Koszul complex on the element $r$. In particular $M\otimes_RK$ is the mapping cone of the map $M\xrightarrow{r}M$. Since $r$ is not a zero-divisor on $M$, the natural surjection 
\[
(M\otimes_RK)\lra M/rM
\]
is a quasi-isomorphism. Thus applying $-\otimes_R K$ to the map  \eqref{eq:ddual} gives the map
\[
\theta(M/rM) \colon M/rM \lra \RHom_{\op A}(\RHom_A(M/rM,A),A)\,.
\]
As $\dim_R(M/rM) = \dim_RM -1$, one has that $\gpdim_A(M/rM)$ is finite, by the induction hypothesis.  Thus the map above is  a quasi-isomorphism, that is to say
\[
\mathrm{cone}(\theta(M)) \otimes K \simeq 0\,.
\]
Observe that the source and target of $\theta(M)$ are complexes that have finitely generated cohomology in each degree. Since the complex above is the mapping cone of the morphism
\[
\mathrm{cone}(\theta(M))  \xrightarrow{\ r\ } \mathrm{cone}(\theta(M)) 
\]
we deduce that the induced map  
\[
\mathrm{H}_*(\mathrm{cone}(\theta(M)))\xrightarrow{\ r \ } \mathrm{H}_*(\mathrm{cone}(\theta(M)))
\]
is an isomorphism. Since $r$ is in $\mathfrak{m}$, Nakayama's Lemma yields that the homology of $\mathrm{cone}(\theta(M))$ is zero, as desired.
\end{proof}

Here is an immediate consequence of the preceding result.

\begin{corollary}
\label{cor:gpd-gor}
Let $A$ be a semilocal Noether algebra with Jacobson radical $J$. If $\gpdim_A\Omega^n(A/J)$ is finite for some integer $n\ge 0$, then the ring $A$ is Iwanaga-Gorenstein.\qed
\end{corollary}

One gets also the following symmetry property of the Jacobson radical; confer Proposition~\ref{prp:gsc}.

\begin{corollary}
\label{cor:gpd}
When $A$ is a semilocal Noether algebra with Jacobson radical $J$ one has an equality
\[
\gpdim_A(J)=\gpdim_{\op A}(J)\,.
\]
\end{corollary}

\begin{proof}
Given Theorem~\ref{thm:gpd} it remains to note that when $A$ is Iwanaga-Gorenstein, one has $\gpdim_A(J_A)=\injdim A-1$, by \cite[Proposition~2.18]{Holm:2004}.
\end{proof}

\subsection*{Commutative rings}

The result below establishes ~\eqref{conj:1} for commutative rings; in this context see the problem posed in \cite[Remark~6.2.16]{Christensen:2000}

\begin{proposition}
\label{prp:commutative}
When $A$ is commutative  noetherian semilocal ring
\[
\ginjdim_AJ=\injdim_AA\,,
\]
where $J$ is the Jacobson radical of $A$.
\end{proposition}

\begin{proof}
The inequality $\ginjdim_AJ\le \injdim_AA$ always holds; see Lemma~\ref{lem:Gid}. The main task is to prove that when $\ginjdim_AJ$ is finite, so is $\injdim_AA$; the stated equality is then well-known; see \cite[Theorem~6.2.15]{Christensen:2000}. Moreover, $A$ is semilocal, it suffices to prove that $\injdim_{A}A_{\mathfrak m}$ is finite for each maximal ideal $\mathfrak m$ of $A$, that is to say, that the local ring $A_{\mathfrak m}$ is Gorenstein. 

Fix a maximal ideal $\mathfrak m$ and let $K$ be the Koszul complex on a finite generating set for the ideal $\mathfrak m$. Let $E$ the injective hull of $A/\mathfrak m$. Since $E$ is artinian, so are the $A$-modules $\mathrm{H}_i(K\otimes_AE)$. Moreover  one has that
\[
\mathfrak m\cdot \mathrm{H}_*(K\otimes_AE) =0\,;
\]
see by \cite[Proposition~1.6.5]{Bruns/Herzog:1998}. In particular, the $A$-modules $\mathrm{H}_i(K\otimes_AE)$ have finite length. Moreover, with $n$  the size of the chosen generating set for $\mathfrak m$, it is clear that
\[
\mathrm{H}_n(K\otimes_AE)=\{x\in E\mid \mathfrak{m}\cdot x=0\}\,,
\]
by  the structure of the Koszul complex; see also the proof of \cite[Theorem~1.6.16]{Bruns/Herzog:1998}. Since each element of $E$ is annihilated by some power of $\mathfrak m$, see, for instance, \cite[Lemma~3.2.7]{Bruns/Herzog:1998}, we deduce $\mathrm{H}_n(K\otimes_AE)\ne 0$. To complete the proof, it remains to observe that the complex $K\otimes_AE$ has finite injective dimension and also finite projective dimension, for then \cite[Proposition~2.10]{Foxby:1979} can be invoked to conclude that $A_{\mathfrak m}$ is Gorenstein.

By construction, the $A$-complex $K\otimes_AE$ is bounded and consists of injective modules so it has finite injective dimension. By the same token, since $\ginjdim_AJ$ is finite, one gets 
\[
\ext^i_A(K\otimes_AE,J)=0 \quad\text{for $i\gg 0$}\,;
\]
see \cite[Theorem~2.22]{Holm:2004}. Thus from Proposition~\ref{prp:pd}---see also \ref{ch:complexes}---we deduce that the $A$-complex $K\otimes_AE$ has finite projective dimension. 
\end{proof}

Next we turn our focus to Artin algebras. 

\section{Artin algebras}
\label{sse:artin}
Let $A$ be an Artin $R$-algebra, that is to say, $R$ is a commutative artinian ring and $A$ is finite $R$-algebra. We set
\[
DA\colonequals \Hom_R(A, E)\,,
\]
where $E$ is the minimal injective cogenerator of $R$.  By \cite[Theorem~2.22]{Holm:2004} any $A$-module $M$ has the property that
\begin{equation}
\label{eq:artin-Gid}
\ginjdim_AM\geq  \sup\{0, i\mid \ext^i_{A}(DA, M)\neq 0\}\,;
\end{equation}
equality holds if $\ginjdim_AM<\infty$.

\begin{lemma}
\label{lem:artin-Gid}
Any Artin algebra $A$ satisfies
\[
\ginjdim_AJ \ge \injdim_{\op A}A
\]
and equality holds when $\ginjdim_AJ$ is finite
\end{lemma}

\begin{proof}
The inequality is the concatenation of (in)equalities
\begin{align*}
\ginjdim_AJ 
	&\ge \sup\{0,i \mid \ext^{i}_{A}(DA, J)\neq 0\}\\
	&=\pdim_A(DA) \\
	&=\injdim_{\op A}A\,,
\end{align*}
where the first one is by \eqref{eq:artin-Gid}, the second from Proposition~\ref{prp:pd}, and the last one is well known.  If $\ginjdim_AJ$ is finite, the inequality above becomes an equality.
\end{proof}

\begin{proposition}
\label{prp:gsc}
Let $A$ be an Artin algebra. The  statements below are equivalent:
\begin{enumerate}[\quad\rm(1)]
\item $\injdim_AA=\injdim_{\op A}A$;
 \item $\ginjdim_AJ=\injdim_AA$ and $\ginjdim_{\op A}J=\injdim_{\op A}A$.
 \end{enumerate}
When they hold, $\ginjdim_AJ=\ginjdim_{\op A}J$.
\end{proposition}

\begin{proof}

(1)$\Rightarrow$(2): Set $d\colonequals \injdim_AA$. If $d<\infty$, the hypothesis means that $A$ is $d$-Gorenstein, so the result follows from Lemmas~\ref{lem:artin-Gid} and ~\ref{lem:Gid}(2).  If $d=\infty$, it follows from the inequalities in Lemma~\ref{lem:artin-Gid}.

(2)$\Rightarrow$(1):  The hypothesis and  Lemma~\ref{lem:artin-Gid} yield inequalities
\[
\injdim_AA\geq \injdim_{\op A}A\quad \text{and}\quad \injdim_{\op A}A\geq \injdim_AA\,.
\]
The required implication follows.
\end{proof}

\begin{chunk}
Let $A$ be an Artin algebra $A$. The finitistic dimension conjecture is that the supremum of the projective dimension of finitely generatd $A$-modules with finite projective dimension is finite. The Gorenstein symmetry conjecture is that equality (1) in Proposition~\ref{prp:gsc} holds.  It is known that if $\injdim_AA$ is finite, then $\injdim_{\op A}A$ is finite if and only if the finitistic dimension conjecture holds for $A$; see \cite[Proposition~6.10]{Auslander/Reiten:1991}. Thus, the Gorenstein symmetry conjecture is a consequence of the finitistic dimension conjecture.
\end{chunk}

Now we move on to results on the injective dimension of higher syzygies of $A_0$.

\begin{chunk}
Let $M$ be a finitely generated $A$-module, and
\[
0 \lra M \lra I^0 \lra I^1 \lra \cdots
\]
its minimal injective coresolution. The  \emph{dominant dimension} of $M$ is 
\[
\domdim_AM \colonequals \inf\{n\mid \text{$I^n$ is not projective}\}.
\]
Note that if $M$ is not projective-injective, then
\begin{equation}
\label{eq:domdim-injdim}
\domdim_AM \le \injdim_AM\,.
\end{equation}
The \emph{codominant dimension}, denoted $\codomdim_AM$, of $M$ is defined dually, in terms of the projective resolution of $M$. One has an equality
\[
\codomdim_A M = \domdim_{\op A} D_A(M)\,.
\]
The dominant dimension of $A$ is defined as the dominant dimension of the regular module $A$. It is well known that $\domdim A= \domdim A^{op}$. 
\end{chunk}

An algebra $A$ is \emph{minimal Auslander-Gorenstein} if $A$ is Gorenstein and
\[ 
\injdim A \leq \domdim A\,.
\]
In \cite{Iyama/Solber:2018}, where this notion is introduced, it is required that $\domdim A\ge 2$, but this is only really necessary to obtain an Auslander type correspondence with precluster tilting objects and not relevant for most theoretical results, so we drop it; this is in line with \cite{Chan/Iyama/Marczinzik:2022}.

Minimal Auslander-Gorenstein algebras are a subsclass of Auslander-Gorenstein rings $A$ introduced by Auslander. The latter are defined via the condition that the minimal injective coresolution $I^{\bullet}$ of $A$ is such that the flat dimension of $I^i$ is at most $i$; see for example \cite{Foxby/Griffith/Reiten:1975}. Since the flat dimension of $I^d$ is equal to $d$, by \cite[Corollary 7]{Iwanaga/Sato:1996}, the minimality in the name "minimal Auslander-Gorenstein" stems from the fact that these are exactly the Auslander-Gorenstein algebras where the flat dimensions of $I^i$ for $i<d$ can be as small as possible, namely zero.

Examples of minimal Auslander-Gorenstein algebras include selfinjective algebras, higher Auslander algebras (which are  the minimal Auslander-Gorenstein algebras of finite global dimension and are in bijective correspondence with cluster-tilting modules~\cite{Iyama:2007}) and centraliser algebras of  matrices~\cite{Cruz/Marczinzik:2021, Xi/Zhang:2022}. When $A$ is selfinjective one has $\domdim A=\infty$; if $A$ is minimal Auslander-Gorenstein, but not selfinjective, then \eqref{eq:domdim-injdim} yields
\[
1\le \injdim_AA=\domdim A\,.
\]

We denote by $\fpdim A$ the finitistic projective dimension of $A$ and by $\fidim$ the finitistic injective dimension of $A$.

\begin{proposition} 
\label{pr:domdim}
Let $A$ be an algebra and $M$ an $A$-module.
\begin{enumerate}[\quad\rm(1)]
\item If $M$ has finite projective dimension and is not projective-injective, then
\[
\fpdim A \geq \pdim M + \domdim M\,.
\]
\item 
If $M$ has finite injective dimension and is not projective-injective, then
\[ 
\fidim A \geq \pdim M + \codomdim M\,.
\]
\item $\pdim M + \domdim M \geq \domdim A$.
\item $\injdim M + \codomdim M \geq \domdim A$.
\end{enumerate}
\end{proposition}

\begin{proof}
We prove (1) and (3); the proofs of (2) and (4) are analogous.

(1) Set $r\colonequals \pdim_AM$. If $\domdim M$ is infinite then the module $\Omega^{-p}(M)$ have projective dimension equal to $p+r$ for arbitrary $p \geq 1$ and thus $\fpdim A$ is infinite.
Now assume that $\domdim M$ is finite and equal to $u$. Then the module $\Omega^{-u}(M)$ has finite projective dimension equal to $u+r$ and thus $\fpdim A$ is larger than or equal to $u+r=\pdim M+ \domdim M$. \newline
Now we show (3). Let 
$$0 \rightarrow Y_n \rightarrow \cdots \rightarrow Y_1 \rightarrow Y_0 \rightarrow M \rightarrow 0$$ 
be a minimal projective resolution of $M$ so that $M$ has projective dimension equal to $n$.
When $0 \rightarrow Y_i \rightarrow I_i^0 \rightarrow I_i^1 \rightarrow \cdots $ is an injective coresolution of $Y_i$ for $0 \leq i \leq n$, then by Miyachi~\cite[Corollary 1.3]{Miyachi:2000}, the module $M$ has an injective coresolution of the form
\[
0 \rightarrow M \rightarrow Q \rightarrow \bigoplus\limits_{i=0}^{n}{I_i^{i+1}} \rightarrow \bigoplus\limits_{i=0}^{n}{I_i^{i+2}} \rightarrow \cdots 
\]
where $Q$ is a direct summand of $\bigoplus\limits_{i=0}^{n}{I_i^i}$.
Let $\domdim A=s$.
Since the $Y_i$ are projective they all have dominant dimension at least $s$ and thus $\domdim M$ is at least $\domdim A - \pdim M$.
\end{proof}

The following corollary can be seen as a non-commutative analogue for minimal Auslander-Gorenstein algebras of the classical Auslander-Buchsbaum formula in commutative algebra~\cite[Theorem~1.3.3]{Bruns/Herzog:1998}.

\begin{corollary} 
\label{co:domdim}
Let $A$ be a connected, minimal Auslander-Gorenstein algebra and $M$ a finitely generated $A$-module that is not projective-injective.
If $M$  has finite projective dimension,  then 
\begin{align*}
&\pdim_AM+ \domdim_AM= \domdim A\,,\\
&\injdim_AM+ \codomdim_AM= \domdim A\,.
\end{align*}
\end{corollary}

\begin{proof}
We prove the equality involving projective dimension; applying it to $DM$ yields the other one. If $A$ is selfinjective, every module of finite projective dimension is  projective-injective, so there is nothing to prove.
Since $A$ is Gorenstein, one gets the first two equalities below:
\[
\fpdim A=\fidim A =\injdim A=\domdim A\,.
\]
The last one holds as $A$ is connected. Proposition~\ref{pr:domdim} gives the desired equality.
\end{proof}

The result below is a positive answer to Question~\ref{question:syz} for the class of minimal Auslander-Gorenstein algebras.

\begin{theorem}
\label{thm:AG-algebras}
Let $A$ be a minimal Auslander-Gorenstein algebra. Then 
\[
\injdim_A\Omega^n(A_0)=\gldim A
\]
for all $n$ such that $\Omega^n(A_0)\ne0$.
\end{theorem}

\begin{proof}
It suffices to prove the stated equality for each block, so in the remainder of the proof we assume $A$ is connected. By Proposition~\ref{prp:syz-question} there is nothing to prove when $A$ has infinite global dimension, for $A$ is Gorenstein.  The result is trivial if $A$ is selfinjective, so we  assume $A$ is not selfinjective and $d\colonequals \gldim A$ is finite. 

Since $A$ is connected, $\domdim A=d=\injdim A$. Let $P$ be an indecomposable projective $A$-module that is not injective and $S$ its top. Then $S$ has injective dimension at least $d$. 

Indeed, consider the projective resolution of $D(A)$:
\[
0 \rightarrow L_d \rightarrow \cdots \rightarrow L_0 \rightarrow D(A) \rightarrow 0
\]
Then the $L_i$ are projective-injective for $0\le i\le d-1$ since $A$ is higher Auslander. If $\injdim S<d$, then $\ext_A^r(D(A),S) \neq 0$ for some $r<d$. But this implies that $P$ is a direct summand of $L_r$  and so is injective, contradicting our assumption on $P$.

Since $\injdim_AS\ge d$, one gets that $\ext_A^n(A_0,S) \neq 0$ for $0\le n\le d$. This means that in the projective cover $P_n$ of $\Omega^n(A_0)$ the module $P$ appears at least once and so the codominant dimension of $\Omega^n(A_0)$ is zero. From Corollary~\ref{co:domdim} one gets
\begin{align*}
\injdim \Omega^n(A_0) 
	&=\injdim_A \Omega^n(A_0)+  \codomdim_A \Omega^n(A_0) \\
	& =\domdim A \\
	&= \injdim A\\
	&=\gldim A\,. \qedhere
\end{align*}
\end{proof}

\end{document}